\documentclass[11pt]{amsart}

\usepackage[vmargin=2.5cm, hmargin=2cm]{geometry}
\usepackage{enumerate}
\usepackage{color}

\DeclareMathOperator{\aff}{aff}
\DeclareMathOperator{\cone}{cone}
\DeclareMathOperator{\G}{G}
\DeclareMathOperator{\F}{F}
\DeclareMathOperator{\Ap}{Ap}

\DeclareMathOperator{\relintr}{\mathrm{relint}}

\DeclareMathOperator{\HH}{H}
\DeclareMathOperator{\Betti}{Betti}
\DeclareMathOperator{\type}{t}

\newtheorem{theorem}{Theorem}
\newtheorem{lemma}[theorem]{Lemma}
\newtheorem{corollary}[theorem]{Corollary}
\newtheorem{proposition}[theorem]{Proposition}

\theoremstyle{definition}
\newtheorem{example}[theorem]{Example}

\theoremstyle{remark}
\newtheorem{remark}[theorem]{\bf Remark}

\title{Frobenius vectors, Hilbert series and gluings}
\author{A. Assi}
\address{Universit\'e d'Angers, D\'epartement de Math\'ematiques, LAREMA, UMR 6093, 2 bd Lavoisier, 49045 Angers Cedex 01, France}
\email{assi@univ-angers.fr}

\thanks{The first author is partially supported by the project GDR CNRS 2945}

\author{P. A. Garc\'{\i}a-S\'anchez}
\address{Departamento de \'Algebra, Universidad de Granada, E-18071 Granada, Espa\~na}
\email{pedro@ugr.es}

\thanks{The second author is supported by the projects MTM2010-15595, FQM-343,  FQM-5849, and FEDER funds.} 

\author{I. Ojeda}
\address{Departamento de Matem\'{a}ticas, Universidad de Extremadura,  E-06071 Badajoz, Espa\~na}
\email{ojedamc@unex.es}

\thanks{The third author is partially supported by the project  MTM2012-36917-C03-01, National Plan I+D+I and by Junta de Extremadura (FEDER funds)}



\begin{document}

\date{\today}

\begin{abstract}
Let $S_1$ and $S_2$ be two affine semigroups and let $S$ be the gluing
of $S_1$ and $S_2$. Several invariants of $S$ are then related to those of
$S_1$ and $S_2$; we review some of the most important properties preserved under gluings. 
The aim of this paper is to prove that this is the case for the Frobenius vector and the Hilbert series. 
Applications to complete intersection affine semigroups are also given.
\end{abstract}

\maketitle

\section{On gluins of affine semigroups}

In this section we take a quick tour summarizing some of the more relevant results on the gluing of affine semigroups. We also introduce concepts and notations that will be used later on in the paper.

An \emph{affine semigroup} $S$ is finitely generated submonoid of $\mathbb Z^m$ for some positive integer $m$. If $S \cap (-S) = 0$, that is to say $S$ is reduced, it can be shown that it has a unique minimal system of generators (see for instance \cite[Chapter 3]{RGS99}). The cardinality of the minimal generating system of $S$ is known as the \emph{embedding dimension} of $S$. Recall that each reduced affine semigroup can be embedded into $\mathbb{N}^m$ for some $m$.  In the following we will assume that our affine semigroups are submonoids of $\mathbb{N}^m$.

Given an affine semigroup $S\subseteq \mathbb N^m$, denote by $\G(S)$ the group spanned by $S$, that is,
\[\G(S)=\big\{ \mathbf z \in \mathbb Z^m \mid \mathbf z= \mathbf a - \mathbf b, \mathbf a, \mathbf b\in S \big\}.\]

Let $A$ be the minimal generating system of $S$, and $A=A_1\cup A_2$ be a nontrivial partition of $A$. Let $S_i=\langle A_i\rangle$ (the monoid generated by $A_i$), $i\in \{1,2\}$. Then $S=S_1+S_2$. We say that $S$ is the \emph{gluing} of $S_1$ and $S_2$ by $\mathbf d$ if
\begin{itemize}
\item $\mathbf d\in S_1\cap S_2$ and,
\item $\G(S_1)\cap \G(S_2) = \mathbf d\mathbb Z$.
\end{itemize} 
We will denote this fact by $S=S_1+_{\mathbf d} S_2$.

There are several properties that are preserved under gluings, and also some invariants of a gluing $S_1+_\mathbf d S_2$ can be computed by knowing their values in $S_1$ and $S_2$. We summarize some of them next.

Assume that $A=\{\mathbf a_1,\ldots,\mathbf a_k\}$. The monoid homomorphism $\varphi: \mathbb N^k\to S$ induced by $\mathbf e_i\mapsto \mathbf a_i$, $i\in\{1,\ldots,k\}$ is an epimorphism (where $\mathbf e_i$ is the $i$th row of the $k\times k$ identity matrix). Thus $S$ is isomorphic as a monoid to $\mathbb N^k/\ker\varphi$, where $\ker\varphi$ is the kernel congruence of $\varphi$, that is, the set of pairs $(\mathbf a,\mathbf b)\in\mathbb N^k\times \mathbb N^k$ with $\varphi(\mathbf a)=\varphi(\mathbf b)$. A \emph{presentation} of $S$ is a system of generators of $\ker\varphi$. A \emph{minimal presentation} is a presentation such that none of its proper subsets is a presentation. All minimal presentations have the same (finite) cardinality (see for instance \cite[Corollary 9.5]{RGS99}). Suppose that $S=S_1+_\mathbf d S_2$, with $S_i=\langle A_i\rangle$, $i\in\{1,2\}$ and $A=A_1\cup A_2$ a nontrivial partition of $A$. We may assume without loss of generality that $A_1=\{\mathbf a_1,\ldots,\mathbf  a_l\}$ and $A_2=\{\mathbf a_{l+1},\ldots,\mathbf  a_k\}$. According to \cite[Theorem 1.4]{gluing}, if we know minimal presentations $\rho_1$ and $\rho_2$ of $S_1$ and $S_2$, respectively, then \[\rho=\rho_1\cup\rho_2\cup\{(\mathbf a,\mathbf b)\}\] is a minimal presentation of $S$, for every $(\mathbf a,\mathbf b)\in \mathbb N^k\times \mathbb N^k$ with $\varphi(\mathbf a)=\varphi(\mathbf b)$,  the first $l$ coordinates of $\mathbf b$ equal to zero and the last $k-l$ coordinates of $\mathbf a$ equal to zero (actually, \cite[Theorem 1.4]{gluing} asserts that this characterizes that $S=S_1+_\mathbf d S_2$).

For an affine semigroup $S$ define $\Betti(S)$ as the set of $\mathbf s\in S$ for which there exists $\mathbf a,\mathbf b\in \varphi^{-1}(\mathbf s)$ such that $(\mathbf a,\mathbf b)$ belongs to a minimal presentation of $S$. Theorem 10 in \cite{uniquely} states that 
\[\Betti(S_1+_\mathbf d S_2)=\Betti(S_1)\cup\Betti(S_2)\cup\{\mathbf d\}.\]
Since several invariants as the catenary degree and the maximum of the delta sets depend on the Betti elements of $S$ (\cite{cat-tame} and \cite{delta}, respectively), the computation of these invariants for $S_1+_\mathbf d S_2$ can be performed once we know their values for $S_1$, $S_2$ and $\mathbf d$ (see for instance \cite[Corollary 4]{acpi}).

Affine semigroups with a single Betti element can be characterized as a gluing of several copies of affine semigroups with empty minimal presentation (and thus isomorphic to $\mathbb N^t$ for some positive integer $t$) along this single Betti element (\cite{single-betti}).

We say that $S$ is \emph{uniquely presented} if for every two minimal presentations $\sigma$ and $\tau$ and every $(\mathbf a,\mathbf b)\in \sigma$, either $(\mathbf a,\mathbf b)\in \tau$ or $(\mathbf b,\mathbf a)\in \tau$, that is, there is a unique minimal presentation up to rearrangement of the pairs of the minimal presentation. It is known (\cite[Theorem 12]{uniquely}) that $S_1+_\mathbf d S_2$ is uniquely presented if and only if $S_1$ and $S_2$ are uniquely presented and $\pm(\mathbf d-\mathbf a)\not\in S_1+_\mathbf d S_2$ for every $\mathbf a\in \Betti(S_1)\cup\Betti(S_2)$. 

It is well known that the cardinality of any minimal presentation of an affine semigroup is greater than or equal to its embedding dimension minus the dimension of the vector space spanned by the semigroup. An affine semigroup is a \emph{complete intersection} affine semigroup if the cardinality of any of its minimal presentations attains this lower bound. It can be shown that an affine semigroup is a complete intersection if and only if it is either isomorphic to $\mathbb N^t$ for some positive integer $t$ or it is the gluing of two complete intersection affine semigroups (\cite{fischer97}). This result generalizes \cite{ci-simplicial} which generalizes the classical result by Delorme for numerical semigroups (\cite{delorme}; actually the definition of gluing was inspired in that paper).

A \emph{numerical semigroup} is a submonoid of $\mathbb N$ with finite complement in $\mathbb N$. It is easy to see that every numerical semigroup is finitely generated (see for instance \cite[Chapter 1]{ns-book}) and thus every numerical semigroup is an affine semigroup. Let $S$ be a numerical semigroup. The largest integer not belonging to $S$ is known as its \emph{Frobenius number}, $\F(S)$. By definition $\F(S)+1+\mathbb N\subseteq S$. This is why the integer $\F(S)+1$ is known as the \emph{conductor} of $S$. Delorme in \cite{delorme} shows that the conductor of a numerical semigroup that is a gluing, say $S_1+_d S_2$, can be computed in terms of the conductors of $S_1$, $S_2$ and $d$. Thus a formula for the Frobenius number of a numerical semigroup that is a gluing is easily derived (this idea is exploited in \cite{ags} to give a procedure to compute the set of all complete intersection numerical semigroups with given Frobenius number). One of the aims of this paper is to generalize this formula for affine semigroups.

Let $S$ be a numerical semigroup. An element $g\in \mathbb Z\setminus S$ is a \emph{pseudo-Frobenius number} if $g+(S\setminus\{0\})\subseteq S$. In particular $\F(S)$ is always a pseudo-Frobenius number. The cardinality of the set of pseudo-Frobenius numbers is known as the (Cohen-Macaulay) \emph{type} of $S$, $\type(S)$. A numerical semigroup is \emph{symmetric} if its type is one (there are plenty of characterizations of this property, see for instance \cite[Chapter 3]{ns-book}). Delorme in his above mentioned paper \cite{delorme} also proved that a numerical semigroup that is a gluing $S_1+_d S_2$ is symmetric if and only if $S_1$ and $S_2$ are symmetric. Nari in \cite[Proposition 6.6]{nari} proved that for a numerical semigroup of the form $S_1+_d S_2$, 
\[\type(S_1+_d S_2)=\type(S_1)\type(S_2)\]
(actually the definition of gluing for numerical semigroups is slightly different and we have to divide $S_1$ and $S_2$ by their greatest common divisors in order to get $S_1$ and $S_2$ numerical semigroups; see the paragraph after Theorem \ref{th-gl-hil}). This formula can be seen as a generalization of the fact that the gluing of symmetric numerical semigroups is again symmetric, and it also shows that 
\begin{itemize}
\item the gluing of pseudo-symmetric numerical semigroups (the only pseudo-Frobenius numbers are the Frobenius number and its half) cannot be pseudo-symmetric,
\item the gluing of two nonsymmetric almost symmetric numerical semigroup is not almost symmetric ($S$ is almost symmetric if the cardinality of $\mathbb N\setminus S$ equals $(\F(S)+\type(S))/2$).
\end{itemize}

Let $S$ be an affine semigroup, and let $\mathbf s\in S\setminus\{0\}$. The \emph{Ap\'ery} set of $\mathbf s$ in $S$ is the set 
\[\Ap(S,\mathbf s)=\{ \mathbf x\in S\mid \mathbf x-\mathbf s \not \in S\}.
\]
This set has in general infinitely many elements. If $S$ is a numerical semigroup and $s\in S\setminus\{0\}$, then $\Ap(S,s)$ has exactly $s$ elements (one for each congruent class modulo $s$). Let $m$ be the least positive integer belonging to $S$, which is known as the \emph{multiplicity} of $S$, and assume that $S$ is minimally generated by $\{n_1,\ldots, n_k\}$, with $n_1<\cdots<n_k$. Clearly, $n_1=m$ and 
$\Ap(S,m)\subseteq \{ \sum_{i=2}^k a_i n_i\mid a_i\le \alpha_i, i\in\{2,\ldots,k\}\}$, with $\alpha_i=\max\{ k\in \mathbb N\mid kn_i\in \Ap(S,m)\}$. When the equality holds we say that the Ap\'ery set of $S$ is $\alpha$-rectangular. Theorem 2.3 in \cite{dms} shows that every numerical semigroup with $\alpha$-rectangular Ap\'ery set other than $\mathbb N$ can be constructed by gluing a numerical semigroup with the same property and a copy of $\mathbb N$.

For a given affine semigroup $S$ and a field $K$, the \emph{semigroup ring} $K[S]$ is defined as $K[S]=\bigoplus_{s\in S} K t^s$ with $t$ an indeterminate. Addition is performed componentwise and the product is calculated by using distributive law and $t^{s}t^{s'}=t^{s+s'}$ for all $s,s'\in S$. If $S$ is a numerical semigroup, then $K[S]$ is a subring of $K[t]$. Recently (\cite{gsl}), it has been shown that if for every relative $I$ ideal of $K[S_i]$, $i\in \{1,2\}$ generated by two monomials, $I \otimes_{K[S_i]} I^{-1}$ has nontrivial torsion, then the same property holds for $S_1+_d S_2$, solving partly a conjecture stated by Huneke and Wiegand (see \cite{gsl} for details; also the restriction of being generated by just two elements can be removed if we take $S_2$ as a copy of $\mathbb N$).

If $S$ is a numerical semigroup minimally generated by $\{n_1,\ldots,n_k\}$, then $\mathfrak m=(t^{n_1},\ldots, t^{n_k})$ is the unique maximal ideal of the power series ring $R=K[[t^{n_1},\ldots, t^{n_k}]]=K[[S]]$. The Hilbert function of the associated graded ring $\mathrm{gr}_\mathfrak m(R)=\bigoplus_{n\in\mathbb N} \mathfrak m^n/\mathfrak m^{n+1}$ is defined as $n\mapsto \dim_K (\mathfrak m^n/\mathfrak m^{n+1})$. In \cite{ams} it is shown that if the Hilbert functions of the associated graded rings of $K[[S_1]]$ and $K[[S_2]]$ are nondecreasing, then so is the Hilbert function of the associated graded ring of $K[[S_1+_d S_2]]$ when the gluing is a ``nice'' gluing (see \cite[Theorem 2.6]{ams} for details; this nice gluing has been also exploited in \cite{jz}).

Lately, for $T=\langle an_1,an_2,an_3,an_4\rangle+_{ab} \langle b\rangle$, Barucci and  Fr\"oberg have been able to compute the Betti numbers of the free resolution of $K[T]$ in terms of that of $K[S]$, with $S=\langle n_1,n_2,n_3,n_4\rangle$ (\cite{bf}).

\section{Gluings and cones}

Given an affine semigroup $S\subseteq \mathbb N^m$, denote by $\cone(S)$ the cone spanned by $S$, that is,
\[\cone(S)=\big\{ q\, \mathbf a \mid q\in \mathbb Q_{\geq 0}, \mathbf a \in S \big\}.\] Observe that $\cone(S)$ is pointed (the only subspace included in it is $\{0\}$), because $S$ is reduced. 

Clearly, if $A$ is finite and generates $S$, then  
\[\G(S)=\left\{\sum_{\mathbf a\in A} z_{\mathbf a} \mathbf a \mid  z_{\mathbf a}\in \mathbb Z \hbox{ for all } \mathbf a\right\}\hbox{ and } \cone(S)=\left\{\sum_{\mathbf a\in A} q_{\mathbf a} \mathbf a \mid  q_{\mathbf a} \in \mathbb Q_{\geq 0} \hbox{ for all } \mathbf a\right\}.\]

We will write $\aff(S)$ for the affine span of $S$, that is, 
\[\aff(S) = G(S) \otimes_\mathbb{Z} \mathbb{Q}.\]

As usual we use the notation
\[\langle A\rangle= \{\sum_{\mathbf a\in A} n_{\mathbf a}\mathbf a\mid n_{\mathbf a}\in \mathbb N \hbox{ for all } \mathbf a\in A\}
\]
(all sums are finite, that is, if $A$ has infinitely many elements, all but a finite number of $z_\mathbf a$, $q_\mathbf a$ and $n_\mathbf a$ are zero). 

\begin{lemma}\label{int-cone-char}
Let $\mathbf r_1,\ldots,\mathbf r_k, \mathbf r_{k+1}$ and $\mathbf{x} \in \cone(\mathbb N^m) \setminus \{0\}$, for some positive integers $m$ and $k$. If $\cone(\mathbf r_1,\ldots, \mathbf r_k)=\cone(\mathbf r_1,\ldots, \mathbf r_k, \mathbf r_{k+1})$, then the following conditions are equivalent:
\begin{enumerate}[(1)]
\item There exist $q_1,\ldots q_k\in \mathbb Q_{>0}$ such that $x=q_1 \mathbf r_1+\cdots +q_k \mathbf r_k$.
\item There exist $q_1',\ldots, q_{k+1}'\in \mathbb Q_{>0}$ such that $x=q_1'\mathbf r_1+\cdots +q_k' \mathbf r_k+q_{k+1}' \mathbf r_{k+1}$.
\end{enumerate}
\end{lemma}

\begin{proof}
Observe that from the hypothesis, $\mathbf r_{k+1}\in \cone(\mathbf r_1,\ldots, \mathbf r_k)$, and thus there exists $t_1,\ldots, t_k\in \mathbb Q_{\geq 0}$ such that $\mathbf r_{k+1}=t_1 \mathbf r_1+\dots+t_k \mathbf r_k$. From this it easily follows (2) implies (1). 

Assume that there exist $q_1,\ldots q_k\in \mathbb Q_{> 0}$ such that $\mathbf x=q_1 \mathbf r_1+\cdots +q_k \mathbf  r_k$. Let $N\in \mathbb N$ be such that for all $i\in \{1,\ldots,k\}$, $t_i/N<  q_i$ (this is possible since $q_i>0$ for all $i$). Take $q_i'=q_i-t_i/N$ (which is a positive rational number) for all $i\in \{1,\ldots, k\}$, and $q_{k+1}'=1/N$. Then $q_1' \mathbf r_1+\cdots +q_k'\mathbf  r_k+q_{k+1}'\mathbf  r_{k+1}=q_1 \mathbf r_1+\cdots + q_k \mathbf r_k-1/N \mathbf r_{k+1}+1/N \mathbf r_{k+1}= \mathbf x$.
\end{proof}

Given $\mathbf r_1,\ldots, \mathbf r_k\in  \cone(\mathbb N^m) \setminus \{0\}$, we define the \emph{relative interior} of $\cone(\mathbf r_1,\ldots,\mathbf r_k)$ by
\[ \relintr(\cone(\mathbf r_1,\ldots, \mathbf r_k))= \big \{ q_1 \mathbf r_1+\cdots + q_k \mathbf r_k\mid q_1,\ldots, q_k\in \mathbb Q_{> 0} \big\}.\]
Observe that the relative interior of a cone $C$ is the topological interior of $C$ in its affine span, $\aff(\mathbf r_1,\ldots, \mathbf r_k),$ with the subspace topology.

For $A\subseteq \mathbb N^m$, we say that $F$ is a \emph{face} of $\cone(A)$ if $F\neq \emptyset$ and there exists $\mathbf c\in \mathbb Q^m\setminus\{0\}$ such that 
\begin{itemize}
\item $F=\{\mathbf x\in \cone(A)\mid \mathbf c\cdot \mathbf x=0\}$ and 
\item $\mathbf c\cdot \mathbf y\ge 0$ for all $\mathbf y\in \cone(A)$.
\end{itemize}
An element $\mathbf{a}\in A$ is an \emph{extremal ray} of $\cone(A)$ if $\mathbb Q_{\ge 0}\mathbf a$ is a one dimensional face of $\cone(A)$.

Now, according to Lemma \ref{int-cone-char}, if $A$ is the minimal system of generators of an affine semigroup $S\subseteq \mathbb N^m$, then we can say that $\mathbf x\in \relintr(\cone(S))$ if and only if $\mathbf x\in \relintr(\cone(A))$, even if $A$ contains elements that are not extremal rays. We get also the following consequence.

\begin{proposition}\label{int-sum}
Let $A$ be a nonempty subset of $\mathbb N^m$, with $m$ a positive integer. Assume that $A=A_1\cup A_2$ is a nontrivial partition of $A$. Then $\relintr(\cone(A)) = \relintr(\cone(A_1)) + \relintr(\cone(A_2))$.
\end{proposition}
\begin{proof}
Obviously, if $\mathbf x_i \in \relintr(\cone(A_i))$, $i \in\{1,2\}$, then $\mathbf x_1 + \mathbf x_2 \in \relintr(\cone(A))$. Now, consider $\mathbf{x} \in \relintr(\cone(A))$. Without loss of generality we may assume that $\mathbf{x} = \sum_{\mathbf{a} \in A} q_\mathbf{a} \mathbf{a}$ with $q_\mathbf{a} \in \mathbb{Q}_{> 0}$. Thus, by taking $\mathbf{x}_i = \sum_{\mathbf{a} \in A_i} q_\mathbf{a} \mathbf{a}$, we are done.
\end{proof}

Notice that if $S$ is the gluing of $S_1$ and $S_2$ by $\mathbf{d}$, then $$\mathbf{d} \not\in \relintr(\cone(S)) \hbox{ implies } \mathbf{d} \not\in \relintr(\cone(S_1)) \cap \relintr(\cone(S_2)).$$ Otherwise, we may take $\mathbf x_i = (1/2) \mathbf{d},\ i \in\{ 1,2\}$.

\begin{proposition}\label{suma-exteriores}
Let $A$ be a nonempty subset of $\mathbb N^m$, with $m$ a positive integer. Assume that $A=A_1\cup A_2$ is a nontrivial partition of $A$. Let $F$ be a face of $\cone(A)$. Then every $\mathbf x\in F$ can be expressed as $\mathbf x_1+\mathbf x_2$ with $\mathbf x_i$ in a face of $\cone(A_i)$, $i\in\{1,2\}$.
\end{proposition}
\begin{proof}
Let $\mathbf x\in F$. Then there exists $\mathbf c\in \mathbb Q^m\setminus\{0\}$ such that $\mathbf c\cdot \mathbf x=0$ and $\mathbf c\cdot \mathbf y\ge 0$ for all $\mathbf y\in \cone(A)$. Notice that $\cone(A)=\cone(A_1)+\cone(A_2)$. Hence there exists $\mathbf x_i\in \cone(A_i)$, $i\in\{1,2\}$ such that $\mathbf x=\mathbf x_1+\mathbf x_2$. As $\cone(A_i)\subseteq \cone(A)$, $\mathbf c\cdot \mathbf y_i\ge 0$, for $i\in\{1,2\}$ and all $\mathbf y_i\in \cone(A_i)$. Hence $0=\mathbf c\cdot\mathbf x=\mathbf c\cdot\mathbf x_1+\mathbf c\cdot\mathbf x_2$ forces $\mathbf c\cdot \mathbf x_1=\mathbf c\cdot\mathbf x_2=0$. We conclude that $\mathbf x_i$ is in the face $\{\mathbf x\in \mathbf Q^n\mid \mathbf c\cdot\mathbf x=0\}\cap\cone(A_i)$ of $\cone(A_i)$, $i\in\{1,2\}$. 
%
%
\end{proof}

We end this section by giving an affine-geometric characterization of gluings.

\begin{proposition}\label{Prop3}
Let $S$ be an affine semigroup and $\mathbf{d} \in \mathbb{N}^n \setminus \{0\}$. If $S = S_1 +_\mathbf{d} S_2$ then $$\cone(S_1) \cap \cone(S_2) = \mathbf d \mathbb{Q}_{\geq 0}.$$
\end{proposition}

\begin{proof}
By definition, $\mathbf d \in S_1\cap S_2$ and, clearly, $\mathbf d \mathbb{Q}_{\geq 0} \subseteq \cone(S_1) \cap \cone(S_2)$. If $\mathbf{d}' \in \cone(S_1) \cap \cone(S_2)$, then $\mathbf{d}' = \frac{z_1}{t_1} \mathbf{a}_1 = \frac{z_2}{t_2} \mathbf{a}_2$, with $z_1, z_2, t_1, t_2 \in \mathbb{N}$, and $\mathbf a_i \in S_i$, $i \in\{1,2\}$. Hence, $t_1 t_2 \mathbf{d}' \in \G(S_1)\cap \G(S_2) = \mathbf d\mathbb Z$, that is, $\mathbf d' \in \mathbf d \mathbb{Q}_{\geq 0}$.
\end{proof}

The above result may be also obtained as a consequence of \cite[Lemma 4.2]{thoma}. 

Observe that the inverse statement is not true as the following simple example shows. Let $S$ be semigroup generated by the columns of the matrix $$A = \left(\begin{array}{ccc|ccc} 4 & 3 & 2 & 3 & 1 & 0 \\ 0 & 1 & 2 & 3 & 3 & 4 \end{array}\right)$$ and let $S_1$ and $S_2$ be the semigroups generated by the three first and the three last columns of $A$, repectively. In this case, $\mathbf{d} := (6,6)^\top \in S_1 \cap S_2$ and $\cone(S_1) \cap \cone(S_2) = \mathbf d \mathbb{Q}_{\geq 0}.$ However, $S_1$ and $S_2$ cannot be glued by $\mathbf d$ because $\G(S_1) \cap \G(S_2)$ has rank $2$; indeed, $3 (2,2) = 2 (3,3)$ and $(0,4) = -2 (4,0)+2(3,1)+(2,2)$.

\begin{corollary}\label{carac-to-one-dim}
Let $S $ be an affine semigroup minimally generated by $A$. Let $A=A_1\cup A_2$ be a nontrivial partition of $A$, and let $S_i=\langle A_i\rangle$, $i\in\{1,2\}$. Set $V=\aff(S_1) \cap \aff(S_2)$. Then, $S=S_1 +_\mathbf{d} S_2$ for some $\mathbf{d} \in \mathbb{N}^n \setminus \{0\}$, if and only if $V = \mathbf d \mathbb Q$ and $S \cap V = (S_1 \cap V) +_\mathbf d (S_2 \cap V)$ for some $\mathbf{d} \in \mathbb{N}^n \setminus \{0\}$.
\end{corollary}

\begin{proof}
If $S = S_1 +_\mathbf{d} S_2$ for some $\mathbf{d} \in \mathbb{N}^n \setminus \{0\}$, by an argument similar to the given in the proof of Proposition \ref{Prop3}, we have that $V = \mathbf{d} \mathbb{Q}$. Now, since $\mathbf d \in (S_1 \cap V) \cap (S_2 \cap V)$ and $\G(S_1 \cap V) \cap \G(S_2 \cap V) = \G(S_1) \cap \G(S_2) = \mathbf d \mathbb Z$, we conclude that $S \cap V$ is the gluing of $S_1 \cap V$ and $S_2 \cap V$ by $\mathbf{d}$. Conversely, let $V = \mathbf{d} \mathbb{Q}$. Since $\G(S_1) \cap \G(S_2) = \G(S_1 \cap V) \cap \G(S_2 \cap V) = \mathbf d \mathbb Z$ and $\mathbf d \in (S_1 \cap V) \cap (S_2 \cap V) = S_1 \cap S_2$, because $\G(S_1) \cap \G(S_2) \subset V$, we are done.
\end{proof}

Let $S$ be the semigroup generated by the columns of the following matrix 
$$A = \left(\begin{array}{ccc|ccc} 4 & 3 & 2 & 3 & 3 & 3 \\ 0 & 1 & 2 & 3 & 2 & 0 \\ 0 & 0 & 0 & 0 & 1 & 3 \end{array}\right)$$ and let $S_1$ ($S_2$, respectively) be the semigroup generated by the three first (last, respectively) columns of $A$. Clearly, $V = \aff(S_1) \cap \aff(S_2) = (1,1,0)^\top \mathbb{Q}$.
Now, since $S_1 \cap V \cong 2 \mathbb N$, $S_2 \cap V \cong 3 \mathbb N$ and 
$S \cap V \cong 2 \mathbb N +_6 3 \mathbb N$, in the light of the above corollary, we conclude that $S = S_1 +_\mathbf{d} S_2$, with $\mathbf d = (6,6,0)^\top$.

\section{Gluings and Frobenius vectors}

Let $S$ be an affine semigroup. We say that $S$ has a \emph{Frobenius vector} if there exists $\mathbf f\in \G(S)\setminus S$ such that 
\[\mathbf f+\relintr(\cone(S))\cap\G(S)\subseteq S\setminus\{0\}\subseteq S.\]
Notice that $\mathbf f+(\relintr(\cone(S))\cap\G(S))\subseteq S\setminus\{0\}$ is equivalent to $(\mathbf f+\relintr(\cone(S)))\cap\G(S)\subseteq S\setminus\{0\}$, and thus we omit the parenthesis in the above condition.

We are going to prove that if $S_1$ and $S_2$ have Frobenius vectors, then so does $S=S_1+_\mathbf d S_2$. 

\begin{theorem}\label{frob-gluing}
Let $S$ be an affine semigroup. 
Assume that $S=S_1+_\mathbf d S_2$. If $S_1$ and $S_2$ have Frobenius vectors, so does $S$. Moreover, if $\mathbf f_1$ and $\mathbf f_2$ are respectively  Frobenius vectors of $S_1$ and $S_2$, then 
\[\mathbf f=\mathbf f_1+\mathbf f_2+\mathbf d\]
is a Frobenius vector of $S$.
\end{theorem}
\begin{proof}
Let $G_1=\G(S_1)$, $G_2=\G(S_2)$, and $G=\G(S)$. Clearly $G=G_1+G_2$, since $S=S_1+S_2$. 

We start by proving that $\mathbf f\in G \setminus S$. As $\mathbf f_1\in G_1$, $\mathbf f_2\in G_2$ and $\mathbf d\in G_1\cap G_2$, we have $\mathbf f\in G$. Assume that $\textbf f\in S$. Then there exist $\mathbf s_1\in S_1$ and $\mathbf s_2\in S_2$ such that $\mathbf f=\mathbf s_1+\mathbf s_2$. Then $\mathbf f_1+\mathbf d-\mathbf s_1=\mathbf s_2-\mathbf f_2\in G_1\cap G_2=\mathbf d\mathbb Z$. So, we can find $k\in \mathbb Z$ such that $\mathbf f_1+\mathbf d-\mathbf s_1=\mathbf s_2-\mathbf f_2=k\mathbf  d$. If $k\le 0$, then $\mathbf f_2=\mathbf s_2-k\mathbf d\in S_2$, a contradiction. If $k>0$, then $\mathbf f_1=\mathbf s_1+(k-1)\mathbf d\in S_1$, which is also impossible, and this proves that $\mathbf f\not\in S$.

In order to simplify the notation, set $C_1=\relintr(\cone(S_1))$, $C_2=\relintr(\cone(S_2))$ and $C=\relintr(\cone(S))$.
Now let us prove that for all $\mathbf x\in C \cap G$, we have that $\mathbf f+\mathbf x\in S$. Since $\mathbf f+\mathbf x\in G$, there must be $\mathbf g_1\in G_1$ and $\mathbf g_2\in G_2$ such that $\mathbf f+\mathbf x=\mathbf g_1+\mathbf g_2$. In light of Proposition \ref{int-sum}, there exists $\mathbf x_1\in C_1$ and $\mathbf x_2\in C_2$ such that $\mathbf x=\mathbf x_1+\mathbf x_2$. Then $\mathbf f+\mathbf x=\mathbf f_1+\mathbf f_2+\mathbf d+\mathbf x_1+\mathbf x_2=\mathbf g_1+\mathbf g_2$. Let $t\in \mathbb Z_{> 0}$ be such that $\mathbf s_1=t\mathbf x_1\in S_1$ and $\mathbf s_2=t\mathbf x_2\in S_2$. This yields $t\mathbf f_1+t\mathbf d+\mathbf s_1-t\mathbf g_1=t\mathbf g_2-t\mathbf f_2-\mathbf s_2=k\mathbf d$ for some integer $k$. Assume that $k\le 0$. Then $t \mathbf f_1+\mathbf s_1+(t-k)\mathbf d=t\mathbf g_1$, and thus $\mathbf f_1+(\mathbf x_1+\frac{t-k}t \mathbf d)=\mathbf g_1$. Observe that $\mathbf x_1+\frac{t-k}t\mathbf d\in C_1$, which implies that $\mathbf g_1\in S_1$ because $\mathbf f_1$ is a Frobenius vector for $S_1$. 

Let $n$ the maximum nonnegative integer such that $\mathbf g_1-nd\in S_1$. 
Hence $\mathbf g_1-(n+1)\mathbf d=\mathbf  f_1+\mathbf x_1+\frac{t-k}t\mathbf d -(n+1)\mathbf d\not\in S_1$, and consequently $tn+k>0$, since otherwise $\frac{t-k}t-(n+1)\ge 0$ and this would lead to $\mathbf x_1+\frac{t-k}t\mathbf d -(n+1)\mathbf d\in C_1$, yielding $\mathbf g_1-(n+1)\mathbf d\in S_1$, a contradiction. Now, $t\mathbf g_2-t\mathbf f_2-\mathbf s_2+tn\mathbf d=(tn+k)\mathbf d$, which means that $\mathbf g_2+n\mathbf d=\mathbf f_2+\mathbf x_2+\frac{tn+k}t\mathbf  d$. As $\mathbf x_2+\frac{tn+k}t\mathbf  d\in C_2$, and $\mathbf f_2$ is a Frobenius vector for $S_2$, we deduce that $\mathbf g_2+n\mathbf d\in S_2$. Finally $\mathbf f+\mathbf x=\mathbf g_1+\mathbf g_2=(\mathbf g_1-n\mathbf d)+(\mathbf g_2+n\mathbf d)\in S_1+S_2=S$.

If $k\ge 0$, then $t \mathbf f_2+\mathbf s_2+t\mathbf d -t\mathbf g_2=t\mathbf g_1-t\mathbf f_2-\mathbf s_1=-k\mathbf d$, and we repeat the above argument by swapping $\mathbf g_1$ and $\mathbf g_2$.
\end{proof}

If $A$ is a set of positive integers, and $S=\langle A\rangle$, then $T=S/\gcd(A)$ is a numerical semigroup, and $\F(T)=\max(\mathbb N\setminus T)$. It follows easily that $\F(S)=\gcd(A)\F(T)$. Recall that the conductor of $T$ is defined as the Frobenius number of $T$ plus one. Hence Theorem \ref{frob-gluing} generalizes the well known formula for the gluing of two submonoids of $\mathbb N$ (\cite[Proposition 10 (i)]{delorme}).

\begin{lemma}\label{frob-indep}
Let $S$ be an affine semigroup minimally generated by $A$. If $A$ is a set of linearly independent elements, then $\mathbf f=-\sum_{\mathbf a\in A}\mathbf  a$ is a Frobenius vector for $S$.
\end{lemma}

\begin{proof}
Let $\mathbf x \in \relintr(\cone(S)) \cap G(S)$. Then $\mathbf x=\sum_{\mathbf a\in A}q_\mathbf a\mathbf  a =\sum_{\mathbf a\in A} z_\mathbf a\mathbf  a$, with $q_\mathbf a\in \mathbb Q_{> 0}$ and $z_\mathbf a\in \mathbb Z$ for all $\mathbf a$. Since the elements in $A$ are linearly independent, this forces $z_\mathbf a = q_\mathbf a$ for all $\mathbf a$; in particular, $z_\mathbf a - 1 \geq 0$ for all $\mathbf a$.
Hence $\mathbf f+\mathbf x=\sum_{\mathbf a\in A} (z_\mathbf a-1)\mathbf  a \in S$.
\end{proof}

Since every complete intersection affine semigroup has either no relations (free  in the categorical sense, that is, its minimal set of generators is a set of linearly independent vectors) or it is the gluing of two affine semigroups (\cite{fischer97}), we get the following result.

\begin{theorem}
Let $S$ be a complete intersection affine semigroup. Then $S$ has a Frobenius vector.
\end{theorem}

\begin{remark}\label{free}
Let $S=S_1+_\mathbf dS_2$ be the gluing of $S_1$ and $S_2$ by $\mathbf d$, and assume that that $S_2=\langle\mathbf  v\rangle$. Hence $\mathbf d=\theta\mathbf  v$ for some $\theta\in {\mathbb N}$. Clearly $-\mathbf v$ is a Frobenius vector for $S_2$ (Lemma \ref{frob-indep}), and if $S_1$ has a Frobenius vector $\mathbf f_1$, then the formula of Theorem 3 implies that 
$\mathbf f=\mathbf f_1-\mathbf v+\theta\mathbf v=\mathbf f_1+(\theta-1)\mathbf v$ is a Frobenius vector of $S$. More generally let $\mathbf v_1,\ldots,\mathbf v_e$ be a set of ${\mathbb Q}$ linearly independent vectors of ${\mathbb N}^e$. Let $S_0=\langle\mathbf  v_1,\ldots,\mathbf v_e\rangle$, and let $\mathbf v_{e+1},\ldots,\mathbf v_{e+h}$ be a set of vectors of ${\mathbb N}^e\cap \cone(\mathbf v_1,\ldots,\mathbf  v_e)$. Set $S_i=\langle \mathbf v_1,\ldots,\mathbf v_{e+i}\rangle $ for all $1\leq i\leq h$ and assume that $S_i=S_{i-1}+_{\theta_i\mathbf v_i}\langle\mathbf  v_{i}\rangle$ (such semigroups are called free semigrous). A Frobenius vector $\mathbf f_0$ of $S_0$ being $\mathbf f_0=-\sum_{k=1}^e\mathbf v_k$ (Lemma \ref{frob-indep}), it follows that 
\begin{equation}\label{ecu2}
\mathbf f_i=\sum_{j=1}^i(\theta_j-1)\mathbf v_j-\sum_{k=1}^e\mathbf v_k
\end{equation}
is a Frobenius vector of $S_i$. This formula has also been proved by the first author in \cite{free}, and gave the following uniqueness condition: this Frobenius vector $\mathbf f$ is minimal with respect to the order induced by $\cone(S)$, that is, for every other Frobenius vector $\mathbf f'$ of $S$, $\mathbf f'\in \mathbf f+\cone(S)$.

We recall that a reduced affine semigroup $S$ is said to be \emph{simplicial} if there are linearly independent elements
$\mathbf{a}_1, \ldots, \mathbf{a}_n \in S$ such that $\cone(S) = \cone(\mathbf{a}_1, \ldots, \mathbf{a}_n)$. Under this hypothesis, conditions for  the existence and conditions for  uniqueness of a Frobenius vector of $S$ are given in \cite{ponomarenko}.
\end{remark}

The formula (\ref{ecu2}) is a special case of the following general formula for a Frobenius vector of a complete intersection affine semigroup.

\begin{remark}\label{gluing-ci}
Recall that according to \cite{fischer97}, any complete intersection affine semigroup is either generated by a set of linearly independent vectors or it is a gluing of two complete intersection numerical semigroups. Thus, repeating this argument recursively, if $S$ is a complete intersection affine semigroup $A$, then there exists a partition $A_1\cup \cdots \cup A_t=A$ such that $A_i$ are sets of linearly independent vectors and 
\[S=S_1+_{\mathbf d_1} S_2+_{\mathbf d_2}\cdots +_{\mathbf d_{t-1}} S_t,\] 
with $S_i=\langle A_i \rangle$. From Theorem \ref{frob-gluing} and Lemma \ref{frob-indep}, it follows that 
\begin{equation}\label{ecu1}
\sum_{i=1}^{t-1} \mathbf d_i -\sum_{\mathbf a\in A}\mathbf  a
\end{equation}
is a Frobenius vector for $S$.
\end{remark}

Next we show that this Frobenius vector is unique in the sense defined above. 

%

\begin{proposition}\label{f-mas-caras}
Let $S$ be a complete intersection affine semigroup and let $\mathbf f$ be defined as in \eqref{ecu1}. Then for every face $F$ of $\cone(S)$, $(\mathbf f+F)\cap S$ is empty.
\end{proposition}
\begin{proof}
Since either $S$ is free or the gluing of two complete intersection affine semigroups $S_1$ and $S_2$, we proceed by induction. If $S$ is free, then Lemma \ref{frob-indep} asserts that $\mathbf f=-\sum_{\mathbf a\in A}\mathbf  a$, with $A$ the minimal generating set of $S$. Clearly in this case the assertion is true.

Now assume that $S=S_1+_\mathbf d S_2$ for some $\mathbf d\in S_1\cap S_2$. From Theorem \ref{frob-gluing}, $\mathbf f=\mathbf f_1+\mathbf f_2+\mathbf d$, where $\mathbf f_i$, $i\in\{1,2\}$, is also defined by \eqref{ecu1}. By induction hypothesis, for every face $F_i$ of $\cone(S_i)$, $i\in\{1,2\}$, $(\mathbf f_i+F_i)\cap S_i=\emptyset$. 

Assume to the contrary that there exists $\mathbf x\in F$ such that $\mathbf f_1+\mathbf f_2+\mathbf d+\mathbf x\in S$. According to Proposition \ref{suma-exteriores}, there exists $\mathbf x_i\in F_i$, $i\in\{1,2\}$, such that $\mathbf x=\mathbf x_1+\mathbf x_2$, for some face $F_i$ of $\cone(S_i)$. Hence there are $\mathbf s_1\in S_1$ and $\mathbf s_2\in S_2$ such that $\mathbf f_1+\mathbf f_2+\mathbf d+\mathbf x_1+\mathbf x_2= \mathbf s_1+\mathbf s_2$. 
Then $\mathbf f_1+\mathbf x_1-\mathbf s_1=\mathbf s_2-\mathbf f_2-\mathbf d-\mathbf x_2=k \mathbf d$ for some integer $k$. As by induction hypothesis, $\mathbf f_1+\mathbf x_1\not\in S_1$, we deduce $k<0$. Therefore $\mathbf f_2+\mathbf x_2=\mathbf s_2-(k+1)\mathbf d$. But $\mathbf f_2+\mathbf x_2\not\in S_2$, which forces $k+1>0$, or equivalently $k\ge 0$. But this is in contradiction with $k<0$.
\end{proof}

\begin{theorem}
Let $S$ be a complete intersection and let $\mathbf f$ be as in \eqref{ecu1}. Assume that $\mathbf f'$ is another Frobenius vector of $S$. Then $\mathbf f'\in \mathbf f+\cone(S)$.
\end{theorem}
\begin{proof}
Write $\mathbf f=\mathbf a-\mathbf b$ and $\mathbf f'=\mathbf a'-\mathbf b'$ with $\mathbf a,\mathbf a', \mathbf b, \mathbf b'\in S$, and let $\mathbf c\in \relintr(\cone(S))$. Then $\mathbf x=\mathbf f+\mathbf b+\mathbf a'+\mathbf c=\mathbf f'+\mathbf b'+\mathbf a+\mathbf c\in (\mathbf f+\relintr(\cone(S)))\cap(\mathbf f'+\relintr(\cone(S)))$. 

Assume that $\mathbf f'\not\in \mathbf f+\cone(S)$. Then the segment joining $\mathbf f'$ and $\mathbf x$ cuts some face of $\mathbf f+ \cone(S)$. Denote by $\mathbf f+F$ this face and let $\mathbf f+\mathbf y$ be this intersection point ($\mathbf y\in F$ and $F$ is a face of $\cone(S)$). There exists a positive integer $k$ such that $k\mathbf y$ is in $S$, and thus $\mathbf f+k\mathbf y\in \G(S)\cap (\mathbf f+F)$. Notice that  $\mathbf f+\mathbf y=\mathbf f'+\mathbf y'$ for some $\mathbf y'\in \relintr(\cone(S))$. As $\mathbf y\in F$, $(k-1)\mathbf y\in \cone(S)$, and consequently $\mathbf f+k\mathbf y=\mathbf f'+(\mathbf y'+(k-1)\mathbf y)\in \mathbf f'+\relintr(\cone(S))$. Hence  $\mathbf f+k\mathbf y\in (\mathbf f'+\relintr(\cone(S)))\cap \G(S)\subseteq S$, in contradiction with Proposition \ref{f-mas-caras}.
\end{proof}

\section{Gluings and Hilbert series}

The \emph{Hilbert series} of $S$ is the Hilbert series associated to $K[S]$: $\HH(S,\mathbf x)=\sum_{\mathbf s\in S} \mathbf x^\mathbf s$, where for $\mathbf s=(s_1,\ldots,s_m)\in \mathbb N^m$, $\mathbf x^\mathbf s=x_1^{s_1}\cdots x_m^{s_m}$. This map is sometimes known in the literature as generating function of $S$, and it has been shown to be of the form $g(S,\mathbf x)/\prod_{\mathbf a\in A}(1-\mathbf x^\mathbf a)$, with $A$ the minimal generating set of $S$ (see \cite[\S 7.3]{bw}).

The next lemma is a straightforward generalization of (4) in \cite{ra-r}.

\begin{lemma}\label{hilbert-ap}
Let $S$ be an affine semigroup and let $\mathbf m\in S\setminus\{0\}$. Then
\begin{equation}\label{hilbert-apery}
\HH(S,x)=\frac{1}{1-x^\mathbf m}\sum_{\mathbf w\in \Ap(S,\mathbf m)} x^\mathbf w.
\end{equation}
\end{lemma}
\begin{proof}
It follows directly from the definition of $\Ap(S,\mathbf m)$, that for every $\mathbf s\in S$, there exist unique $k\in \mathbb N$ and $\mathbf w\in \Ap(S,\mathbf m)$ such that $\mathbf s=k\mathbf m+\mathbf w$. Hence 
\[\HH(S,\mathbf x)=\sum_{k\in \mathbb N,\mathbf  w\in \Ap(S,\mathbf m)} \mathbf x^{k\mathbf m+\mathbf w}= \sum_{k\in \mathbb N}(\mathbf x^\mathbf m)^\mathbf k\sum_{\mathbf w\in \Ap(S,\mathbf m)}x^\mathbf w.\]
The proof follows by taking into account that $\sum_{k\in \mathbb N}(\mathbf x^\mathbf m)^\mathbf k = 1/(1-\mathbf x^\mathbf m)$.
\end{proof}

The following result can also be understood as a generalization of (4) in \cite{ra-r}, since for simplicial affine semigroups that are Cohen-Macaulay the set $\bigcap_{i=1}^m \Ap(S,\mathbf v_i)$, with $\mathbf v_1,\ldots,\mathbf v_m$ a set of extremal rays of $S$, plays a similar role to the Ap\'ery set of an element in a numerical semigroup (compare \cite[Theorem 1.5]{cm} and \cite[Lemma 2.6]{ns-book}).

\begin{proposition}
Let $S$ be a simplicial affine semigroup with extremal rays $\mathbf v_1,\ldots,\mathbf  v_m$. Then $\HH(S,\mathbf x)=\frac{P(\mathbf x)}{\prod_{i=1}^m(1-\mathbf x^{\mathbf v_i})}$, with $P(\mathbf x)$ a polynomial.
\end{proposition}
\begin{proof}
Let $Ap=\bigcap_{i=1}^m \Ap(S,\mathbf v_i)$. In view of \cite[Section 1]{cm}, this set is finite. Moreover, from \cite[Theorem~1.5]{cm} we know that every element $\mathbf s$ in $S$ can be expressed uniquely as $\mathbf s=\sum_{i=1}^m a_i \mathbf v_i+ \mathbf w$ with $a_1,\ldots, a_d\in \mathbb N$ and $\mathbf w\in Ap$. Arguing as in Lemma \ref{hilbert-ap},
\[ \HH(S,\mathbf x)=\sum_{\mathbf s\in S}\mathbf x^\mathbf s= \frac{\sum_{\mathbf w\in Ap}\mathbf x^\mathbf w}{\prod_{i=1}^m(1-x^{\mathbf v_i})},\]
which concludes the proof.
\end{proof}

\begin{theorem}\label{th-gl-hil}
Let $S$, $S_1$ and $S_2$ be affine semigroups, and let $\mathbf d\in S$. Assume that $S=S_1+_\mathbf d S_2$. Then
\[
\HH(S_1+_\mathbf d S_2,\mathbf x)= (1-\mathbf x^d)\HH(S_1,\mathbf x)\HH(S_2,\mathbf x).
\]
\end{theorem}
\begin{proof}
From \eqref{hilbert-apery},
\[\HH(S,\mathbf x)= \frac{1}{1-\mathbf x^\mathbf d}\sum_{\mathbf w\in \Ap(S,\mathbf d)}\mathbf  x^\mathbf w.\] 

From \cite[Theorem 1.4]{gluing}, the mapping 
\begin{equation}\label{sum-ap}
 \Ap(S_1,\mathbf d)\times \Ap(S_2,\mathbf d)\to \Ap(S,\mathbf d),\ (x,y)\mapsto x+y
\end{equation}
is a bijection, and thus $\Ap(S,\mathbf d)=\Ap(S_1,\mathbf d)+\Ap(S_2,\mathbf d)$.
Hence,
\[ 
\sum_{\mathbf w\in \Ap(S\mathbf ,d)} x^\mathbf w = \sum_{\mathbf w_1\in\Ap(S_1,\mathbf d)} \sum_{\mathbf w_2\in\Ap(S_2,\mathbf d)} \mathbf x^{\mathbf w_1+\mathbf w_2} = \left(\sum_{\mathbf w_1\in \Ap(S_1,\mathbf d)}\mathbf  x^{\mathbf w_1}\right)\left(\sum_{\mathbf w_2\in \Ap(S_2,\mathbf d)}\mathbf  x^{\mathbf w_2}\right).
\] 

As $\HH(S_1,\mathbf x)= \frac{1}{1-\mathbf x^{\mathbf d}}\sum_{\mathbf w_1\in\Ap(S_1,\mathbf d)}x^{\mathbf w_1}$ and $\HH(S_2,\mathbf x)= \frac{1}{1-\mathbf x^{\mathbf d}}\sum_{\mathbf w_2\in\Ap(S_2,\mathbf d)}\mathbf x^{\mathbf w_2}$, we get 
\[
\HH(S,\mathbf x)= (1-\mathbf x^{\mathbf d})\HH(S_1,\mathbf x)\HH(S_2,\mathbf x).\qedhere
\]
\end{proof}

If $S$ is a numerical semigroup ($\gcd(S)=1$), and it is a gluing of $M_1$ and $M_2$, then $S_1=M_1/d_1$ and $S_2=M_2/d_2$ are also numerical semigroups, with $d_i=\gcd(M_i)$, $i\in\{1,2\}$. Hence $S=d_1S_1+_{d_1d_2}d_2S_2$ and $\mathrm{lcm}(d_1,d_2)=d_1d_2$. We say in this setting that $S$ is a gluing of $S_1$ and $S_2$ at $d_1d_2$. 

From the definition of Hilbert series associated to a submonoid $M$ of $N$, it follows easily that if $k\mid \gcd(M)$, then 
\begin{equation}\label{hilbert-cociente}
\HH(M/k,x^k)=\HH(M,x).
\end{equation}

We get the following corollary.

\begin{corollary}\label{cor-gl-hil}
Let $S$ be a numerical semigroup. Assume that $S=d_1S_1+_{d_1d_2}d_2S_2$ is a gluing of the numerical semigroups $S_1$ and $S_2$. Then 
\[\HH(S,x)=(1-x^{d_1d_2})\HH(S_1,x^{d_1})\HH(S_2,x^{d_2}).\]
\end{corollary}

%
%
%
%

\begin{example}
Let $S=\langle a,b\rangle$ with $a$ and $b$ coprime positive integers. Then $S=a\mathbb N+_{ab} b\mathbb N$. Then by Corollary \ref{cor-gl-hil}, \[\HH(\langle a,b\rangle,x)= (1-x^{ab})\HH(\mathbb N, x^a)\HH(\mathbb N, x^b)= \frac{1-x^{ab}}{(1-x^a)(1-x^b)}.\] 

If we do this computation by using the formula $\HH(\langle a,b\rangle,x)=\frac{1}{1-x^a}\sum_{w\in \Ap(\langle a,b\rangle, a)} x^w$, we obtain, $\HH(\langle a,b\rangle,x)=\frac{1}{1-x^a} \sum_{k=0}^{a-1} x^{kb} = \frac{1}{1-x^a}\frac{1-x^{ab}}{1-x^b}$. Observe that this is a particular case of \cite[Proposition 2]{ra-r} (see also \cite[Theorem 4]{m} for a relationship with inclusion-exclusion polynomials). 
\end{example}

This idea can be generalized to any complete intersection affine semigroup. The base setting is the following.

\begin{lemma}
Let $A\subseteq \mathbb N^m$ be a set of linearly independent vectors. Then
\[
\HH(\langle A\rangle,\mathbf x)=\frac{1}{\prod_{\mathbf a\in A}(1-\mathbf x^\mathbf a)}.
\]
\end{lemma}
\begin{proof}
Assume that $A=\{\mathbf a_1,\ldots,\mathbf a_k\}$, and write $S=\langle A\rangle$. Notice that the map $\mathbb N^k\to S$, $(n_1,\ldots,n_k)\mapsto \sum_{i=1}^k n_i\mathbf a_i$ is a monoid isomorphism. Hence 
\[\sum_{\mathbf s\in S}x^s=\sum_{n_1\in \mathbb N, \ldots, n_k\in \mathbb N} \mathbf x^{n_1\mathbf a_1+\cdots +n_k\mathbf a_k}= \prod_{i=1}^k \sum_{n\in \mathbb N} (x^{\mathbf a_i})^n,\]
and the proof follows easily.
\end{proof}

\begin{proposition}
Let $S$ be a free affine semigroup. Assume that 
\[S=(\cdots(\langle \mathbf  v_1,\ldots, \mathbf v_e\rangle+_{\theta_{e+1}\mathbf v_{e+1}}\langle \mathbf v_{e+1}\rangle)+_{\theta_{e+2}\mathbf v_{e+2}}\cdots) +_{\theta_{e+h}\mathbf v_{e+h}}\langle \mathbf v_{e+h}\rangle.\]
Then 
\[ \HH(S,\mathbf x)= \frac{\prod_{i=1}^h (1-\mathbf x^{\theta_{e+i}\mathbf v_{e+i}})}{\prod_{i=1}^{e+h} (1-\mathbf x^\mathbf v_i)}.\]
\end{proposition}

This is indeed a particular case of the following theorem.

\begin{theorem} \label{formula-hilbert-ci}
Let $S$ be a complete intersection affine semigroup minimally generated by $A$. Let $\mathbf d_1,\ldots, \mathbf d_{t-1}$ be as in Remark \ref{gluing-ci},
\[\HH(S,\mathbf x) = \frac{\prod_{i=1}^{t-1}(1-\mathbf x^{\mathbf d_i})}{\prod_{\mathbf a\in A}(1-\mathbf x^\mathbf a)}.\]
\end{theorem}

\begin{remark}
Observe that if we substract the degree of the numerator and denominator of the formula given in Theorem \ref{formula-hilbert-ci} we obtain Formula (\ref{ecu1}).
\end{remark}

\begin{example}
Let $S=\langle 4,5,6\rangle = \langle 4,6\rangle +_{10} 5\mathbb N=(4\mathbb N+_{12}6\mathbb N)+_{10}5\mathbb N$. Then
\[\HH(\langle 4,5,6\rangle,x)=\frac{(1-x^{10})(1-x^{12})}{(1-x^4)(1-x^5)(1-x^6)}.\]
The Frobenius number of $S$ is $10+12-(4+5+6)=7$.
\end{example}

\section*{Acknowledgments}

Part of this research was performed while the second author visited the Universit\'e d'Angers, and he wants to thank the D\'epartement de Math\'ematiques of this university for its kind hospitality


\end{document}